\documentclass[12pt]{amsart}
\usepackage{cite,url,graphicx}
\newcommand{\K}{\mathcal{K}}
\newcommand{\A}{\mathcal{A}}
\newcommand{\N}{\mathbb{N}}
\newcommand{\R}{\mathbb{R}}
\newcommand{\Z}{\mathbb{Z}}
\newcommand{\T}{\mathbb{T}}

\newcommand{\auto}{\mathsf{A}}
\newcommand{\LL}{\mathcal{L}}
\newcommand{\eps}{\epsilon}
\newcommand{\trans}{\mathsf{T}}
\newcommand{\nuhat}{\widehat{\nu}}
\newcommand{\psihat}{\widehat{\psi}}

\newtheorem{theorem}{Theorem}
\newtheorem{proposition}{Proposition}
\newtheorem{lemma}{Lemma}
\theoremstyle{definition}
\newtheorem{remark}{Remark}
\theoremstyle{remark}
\newtheorem{example}{Example}
\newtheorem*{ackno}{Acknowledgement}
\begin{document}
\author[P. J. Grabner]{Peter J. Grabner\textsuperscript{*}}
\thanks{\textsuperscript{*} The author is supported by the Austrian Science
  Fund FWF project F5503 (part of the Special Research Program (SFB)
  ``Quasi-Monte Carlo Methods: Theory and Applications'')}
\address{Institut f\"ur Analysis und Zahlentheorie,
Technische Universit\"at Graz,
Kopernikusgasse 24/II,
8010 Graz,
Austria}
\email{peter.grabner@tugraz.at}
\dedicatory{Dedicated to J\"org Thuswaldner on the occasion of his
  50\textsuperscript{th}  birthday.}
\begin{abstract}
  We study measures that are obtained as push-forwards of measures of maximal
  entropy on sofic shifts under digital maps
  $(x_k)_{k\in\N}\mapsto\sum_{k\in\N}x_k\beta^{-k}$, where $\beta>1$ is a Pisot
  number. We characterise the continuity of such measures in terms of the
  underlying automaton and show a purity result.
\end{abstract}
\title{Purity results for some arithmetically defined measures}
\maketitle
\section{Introduction}\label{sec:introduction}
\nocite{Lothaire2002:algebraic_combinatorics_words}
Digital representations of real numbers by infinite series
\begin{equation}\label{eq:digit-sum}
  \sum_{k=1}^\infty\frac{x_k}{\beta^k}
\end{equation}
with $x_k\in\A$, a finite alphabet, and $\beta>1$ have attracted attention from
different points of view. The underlying dynamical system given by the map
$T_\beta:x\mapsto \beta x\mod1$ has been studied extensively since the seminal
papers \cite{Renyi1957:representation_real_numbers,
  Parry1960:beta_expansions_real}. For an overview of the development we refer
to \cite{Dajani_Kraaikamp2002:ergodic_theory_numbers,
  Frougny2002:numeration_systems, Frougny1992:representations_numbers_finite,
  Beal_Perrin1997:symbolic_dynamics_finite}. The original study was carried out
for the ``canonical'' digit set $\A=\{0,1,\ldots,\lceil\beta\rceil-1\}$, but
many variations have been studied. It turned out in
\cite{Parry1960:beta_expansions_real} that Pisot numbers $\beta$ play a very
important r\^ole in that context, as for these $\beta$ the transformation
$T_\beta$ has especially nice properties. In this case the set of
representations of all real numbers in $[0,1]$ obtained by iteration of
$T_\beta$ is a sofic shift (see \cite{Parry1960:beta_expansions_real}); the
definition of a sofic shift will be given in
Section~\ref{sec:regular-languages}. A Pisot number $\beta=\beta_1$ (of degree
$r\geq1$) is an algebraic integer all of whose Galois conjugates
$\beta_2,\ldots,\beta_r$ have modulus $<1$ (see
\cite{Bertin_Decomps-Guilloux_Grandet-Hugot+1992:pisot_salem_numbers}). Notice
that integers $\geq2$ are also considered as Pisot numbers. Pisot numbers have
the nice property that their powers are ``almost integers'', meaning that
$(\beta^n\mod1)_{n\in\N}$ tends to $0$.

In the present paper we will change the point of view starting with a one-sided
sofic shift space $\K^+\subset\A^{\N}$, where $\A\subset\Z$ is the underlying
set of digits. Then we consider a map $\phi^+:\K^+\to\R$ mapping
$(x_k)_{k\in\N}$ to the series \eqref{eq:digit-sum} for $\beta$ a Pisot
number. Of course, in general nothing can be said about injectivity of this
map, or even the structure of the image. Even for the full shift $\A^\N$ the
structure of the image can be intricate (see
\cite{Winkler2001:order_theoretic_structure}). If $\K^+$ is equipped with a
shift invariant measure, then this measure is pushed forward to $\R$ by the map
$\phi^+$. The properties of measures obtained in this way are our object of
study. For the measure on $\K^+$ we will take the unique shift invariant
measure of maximal entropy, or Parry measure, see
\cite{Parry1966:symbolic_dynamics_transformations,
  Parry1964:intrinsic_markov_chains}. This will be discussed in
Section~\ref{sec:meas-shift-spac}.

Measures of this kind occur in different contexts. Possibly, the earliest
occurrence was in two papers by Erd\H{o}s \cite{Erdoes40,Erdoes39}, where he
proved the singularity of this measure for $\K$ being the full shift
$\{\pm1\}^{\mathbb{N}}$. This is the most studied case and goes under the name
``Bernoulli convolution''. We will discuss that further in
Section~\ref{sec:bern-conv}.

In the context of
studying redundant expansions of integers in the context of fast multiplication
algorithms used in cryptography, the precise study of the number of
representations of an integer $n$ in the form
\begin{equation*}
  n=\sum_{k=0}^Kx_k2^k\quad\text{with }x_k\in\{0,\pm1\}
\end{equation*}
and minimising
\begin{equation*}
  \sum_{k=0}^K|x_k|
\end{equation*}
led to a singular measure on $[-1,1]$ (see
\cite{Grabner_Heuberger2006:number_optimal_base}). Further results in this
direction were obtained in
\cite{Grabner_Heuberger_Prodinger2005:counting_optimal_joint,
  Grabner_Heuberger_Prodinger2004:distribution_results_low}. A more general
point of view replacing the powers of $2$ by the solution of a linear
recurrence has been taken in
\cite{Grabner_Steiner2011:redundancy_of_minimal}. Furthermore, such measures
occur as spectral measures of dynamical systems related to numeration systems
\cite{Grabner_Liardet_Tichy2005:spectral_disjointness_dynamical}, in the study
of diffraction patters of tilings (see
\cite{Baake_Grimm2019:fourier_transform_rauzy,
  Baake_Grimm2013:aperiodic_order_I}), and as spectral measures of substitution
dynamical systems (see \cite{Queffelec2010:substitution_dynamical_systems}). In
Section~\ref{sec:gener-erdh-meas} we will present two main results, namely the
fact that the measure is pure (meaning that the Lebesgue decomposition only has
one term), and a characterisation of continuity of the measure in terms of
properties of the underlying automaton.

Erd\H{o}s' proof of the singularity of the measure in
\cite{Erdoes40,Erdoes39} uses the fact that the Fourier transform of these
measures does not tend to $0$ at $\infty$, and this method was used in many
other cases.  This motivates the study of the Fourier transform of the measures
under consideration. The transform can be expressed in terms of infinite matrix
products, which allow the computation of limits
\begin{equation}\label{eq:limnuhat}
  \lim_{k\to\infty}\nuhat(z\beta^k)
\end{equation}
for $z\in\Z[\beta]$. In Section~\ref{sec:four-transf-matr} we find an
interpretation of these limits as Fourier coefficients of a measure on the
torus given by a two sided version of the map $\phi^+$. A similar map has been
introduced and studied in \cite{Sidorov2001:bijective_general_arithmetic,
  Sidorov_Vershik1998:ergodic_properties_erdoes}. This will be used to show
that the vanishing of the limits \eqref{eq:limnuhat} for all $z\in\Z[\beta]$,
$z\neq0$ is equivalent to absolute continuity for these measures.

Very recently, similar results for self similar measures have been obtained in
\cite{Bremont2021:self_similar_measures}. There for a set of linear maps
$\phi_k:x\mapsto r_kx+b_k$ with $0<r_k< 1$ ($k=1,\ldots,N$) the distribution
measure $\nu$ of the map
\begin{equation*}
  (\omega_1,\omega_2,\ldots)\mapsto
  \lim_{n\to\infty}\phi_{\omega_n}\circ\phi_{\omega_{n-1}}\circ
  \cdots\circ\phi_{\omega_1}(x_0)
\end{equation*}
is studied, where $\{1,\ldots,N\}^{\mathbb{N}}$ is equipped with the infinite
product of the measures $\mathbb{P}(\{k\})=p_k$ ($k=1,\ldots,N$) for any choice
of the vector $(p_1,\ldots,p_N)$ with $p_k>0$ and $\sum_kp_k=1$. It is shown
(see \cite[Theorem~2.3]{Bremont2021:self_similar_measures}) that
$\widehat\nu(t)\not\to0$ for $|t|\to\infty$, if and only if the contraction
factors $r_k$ are all negative integer powers of the same Pisot number. It is
also shown (see \cite[Theorem~2.4]{Bremont2021:self_similar_measures}) that for
these measures $\lim_{|t|\to\infty}\widehat{\nu}(t)=0$ (called the
\emph{Rajchman property}) is equivalent to absolute continuity. This is a
phenomenon very similar to Theorem~\ref{thm:Rajchman} in the context of this
paper.

In a final Section~\ref{sec:examples} we exhibit several simple examples as
applications of our results.

\section{Regular languages and sofic shifts}
\label{sec:regular-languages}
Let $\A\subset\mathbb{Z}$ be a finite alphabet, and denote by $\A^*$ the set of
finite words over $\A$,
i.~e. $\A^*=\{\epsilon\}\cup\bigcup_{n\in\mathbb{N}}\A^n$, where
$\epsilon$ denotes the empty word.  Let $G=(V,E)$ be
a finite directed graph (see~\cite{Diestel2017:graph_theory}) equipped with a
\emph{labelling} $\ell:E\to\A$. Then the pair $(G,\ell)$ is called a
\emph{labelled graph}. A finite automaton is a quadruple $\auto=(G,\ell,I,T)$,
where $I$ (initial states) and $T$ (terminal states) are subsets of the set of
vertices (also called \emph{states} in this context). A \emph{path} of length
$n$ in the graph $G$ is a sequence of edges $p=e_1e_2\ldots e_n$, such that for
every $j=1,\ldots,n-1$ the edges $e_j=(\mathrm{i}(e_j),\mathrm{t}(e_j)))$
satisfy $\mathrm{t}(e_j)=\mathrm{i}(e_{j+1})$; the terminal vertex
$\mathrm{t}(e_j)$ of every edge coincides with the initial vertex
$\mathrm{i}(e_{j+1})$ of the consecutive edge. We say that $p$ connects
$\mathrm{i}(e_1)$ and $\mathrm{t}(e_n)$. The language $\LL=\LL(\auto)$
recognised by the automaton $\auto$ is given by
\begin{equation}
  \label{eq:language}
  \begin{split}
    \LL_n&=\left\{\ell(e_1)\ell(e_2)\ldots\ell(e_n)\,\middle|\, e_1e_2\ldots
      e_n\text{ a path in }G, \mathrm{i}(e_1)\in I, \mathrm{t}(e_n)\in T
    \right\}\\
    \LL&=\{\eps\}\cup\bigcup_{n=1}^\infty\LL_n,
  \end{split}
\end{equation}
where $\eps$ denotes the \emph{empty word}, which by definition has length $0$;
$\LL_n$ is the set of words of length $n$. A subset of $\A^*$ is called a
\emph{regular language}, if it is recognised by a finite automaton $\auto$. A
language $\LL$ is called \emph{irreducible}, if for any $w_1,w_2\in\LL$ there
exists a $w\in\A^*$ such that $w_1ww_2\in\LL$. This is equivalent to the
underlying graph being \emph{strongly connected}; for any two vertices
$v_1,v_2\in V$ there is a path connecting $v_1$ with $v_2$. The language is
called \emph{primitive}, if there is an $N\in\mathbb{N}$ such that for any
$w_1,w_2\in\LL$ and any $n\geq N$ there exists a word $w\in\A^*$ of length $n$
such that $w_1ww_2\in\LL$. 

From now on we assume that all languages are primitive. For a comprehensive
introduction to the theory of formal languages we refer to
\cite{Eilenberg1974:automata_languages_machines,
  Harrison1978:introduction_formal_language,
  Sakarovitch2009:elements_automata_theory}.

For a given automaton $\auto$ we study the sets of one- and two-sided infinite
words recognised by $\auto$
\begin{equation}
  \label{eq:shifts}
  \begin{split}
    \K_I^+&=\left\{(\ell(e_k))_{k\in\mathbb{N}}\,\middle|\, \forall n\in\mathbb{N}:
      e_1e_2\ldots e_n\text{ is a path in }G, \mathrm{i}(e_1)\in I\right\}\\
    \K^+&=\left\{(\ell(e_k))_{k\in\mathbb{N}}\,\middle|\, \forall n\in\mathbb{N}:
      e_1e_2\ldots e_n\text{ is a path in }G\right\}\\
    \K&=\left\{(\ell(e_k))_{k\in\mathbb{Z}}\,\middle|\,\forall m<n: e_me_{m+1}\ldots
      e_n\text{ is a path in }G\right\}.
  \end{split}
\end{equation}
The set $\K$ is called the (two-sided) \emph{sofic} shift associated to
$\auto$ (see~\cite{Lind_Marcus1995:introduction_symbolic_dynamics}), $\K^+$ is
the one-sided sofic shift, both spaces are closed under the according shift
transformation
\begin{align*}
  \sigma^+&:\K^+\to\K^+: (x_k)_{k\in\mathbb{N}}\mapsto
  (x_{k+1})_{k\in\mathbb{N}}\\
  \sigma&:\K\to\K: (x_k)_{k\in\mathbb{Z}}\mapsto (x_{k+1})_{k\in\mathbb{Z}}.
\end{align*}
The space $\K_I^+$, which can be seen as an extension of $\LL$ to infinite
words, is in general not closed under $\sigma^+$, but the relation
\begin{equation*}
  \K^+=\bigcup_{n=0}^N(\sigma^+)^n\K_I^+
\end{equation*}
holds for some $N\in\mathbb{N}$ as a consequence of the strong connectedness of
the graph underlying $\auto$. Notice that $\sigma$ is bijective, whereas
$\sigma^+$ is not.
\section{Measures on shift spaces}\label{sec:meas-shift-spac}
We equip the spaces $\K_I^+$ and $\K^+$ with a ``canonical'' measure that we
will define now. We define the \emph{cylinder set}
\begin{equation*}
  [\varepsilon_1,\varepsilon_1,\ldots,\varepsilon_{k}]=
  \left\{(x_n)_{n\in\mathbb{N}}\in\K_I^+\,\middle|\,
  x_1=\varepsilon_1,\ldots,x_k=\varepsilon_k\right\}
\end{equation*}
for $\varepsilon_i\in\A$ for $i=1,\ldots,k$. The cylinder sets generate a
topology on $\K_I^+$ and also the $\sigma$-algebra of Borel sets for this
topology. We define
\begin{equation}
  \label{eq:muI-def}
  \mu_I^+([\varepsilon_1,\varepsilon_2,\ldots,\varepsilon_{k}])=
  \lim_{n\to\infty}
  \frac{\#([\varepsilon_1,\varepsilon_2,\ldots,\varepsilon_{k}]\cap\LL_n)}
  {\#\LL_n};
\end{equation}
the existence of the limit will become obvious from the following discussion.

For $a\in\A$ define the \emph{$a$-transition matrix} by
\begin{equation}
  \label{eq:Ma-def}
  (M_a)_{ij}=
  \begin{cases}
    1&\text{if }(i,j)\in E\text{ and }\ell((i,j))=a\\
    0&\text{otherwise}
  \end{cases}
\end{equation}
and
\begin{equation}
  \label{eq:M-def}
  M=\sum_{a\in\A}M_a.
\end{equation}
Furthermore, set $\mathbf{v}_I$ and $\mathbf{v}_T$ the indicator vectors of the
sets $I$ and $T$, respectively. Then for $n\geq k$ we have
\begin{equation*}
  \#([\varepsilon_1,\varepsilon_1,\ldots,\varepsilon_{k}]\cap\LL_n)=
  \mathbf{v}_I^\trans M_{\varepsilon_1}\cdots M_{\varepsilon_k}M^{n-k}\mathbf{v}_T.
\end{equation*}
By the assumption that the language $\LL$ is primitive which is equivalent to
the fact that $M$ is primitive, the Perron-Frobenius theorem
(see\cite{Seneta1981:nonnegative_matrices_markov}) implies that there is a
dominating eigenvalue $\lambda>0$ such that
\begin{equation}
  \label{eq:perron}
  M^n=\lambda^n\mathbf{v}_R\mathbf{v}_L^\trans+o(\lambda^n),
\end{equation}
where $\mathbf{v}_L^\trans$ is a left eigenvector of $M$ for the eigenvalue
$\lambda$, and $\mathbf{v}_R$ is a right eigenvector with
\begin{equation*}
  \mathbf{v}_L^\trans\mathbf{v}_R=1.
\end{equation*}
With this we can write the quantity under the limit in \eqref{eq:muI-def} as
\begin{equation*}
  \frac{\mathbf{v}_I^\trans M_{\varepsilon_1}\cdots M_{\varepsilon_k}
    \left(\lambda^{n-k}\mathbf{v}_R\mathbf{v}_L^\trans
      +o(\lambda^n)\right)\mathbf{v}_T}
  {\mathbf{v}_I^\trans\left(\lambda^n\mathbf{v}_R\mathbf{v}_L^\trans
      +o(\lambda^n)\right)\mathbf{v}_T},
\end{equation*}
which shows that the limit exists and equals
\begin{equation}
  \label{eq:muI-limit}
  \mu_I^+([\varepsilon_1,\varepsilon_2,\ldots,\varepsilon_{k}])=
  \lambda^{-k}\frac{\mathbf{v}_I^\trans M_{\varepsilon_1}\cdots
    M_{\varepsilon_k}\mathbf{v}_R}{\mathbf{v}_I^\trans\mathbf{v}_R}.
\end{equation}
On $\K^+$ we define the measure $\mu^+$ by
\begin{equation}
  \label{eq:mu-def}
  \mu^+([\varepsilon_1,\varepsilon_2,\ldots,\varepsilon_{k}])=
  \lambda^{-k}\mathbf{v}_L^\trans M_{\varepsilon_1}\cdots
    M_{\varepsilon_k}\mathbf{v}_R.
\end{equation}
This measure is shift invariant by the observation
\begin{multline*}
  \mu^+((\sigma^+)^{-1}[\varepsilon_1,\varepsilon_2,\ldots,\varepsilon_{k}])=
  \sum_{a\in\A}\mu^+([a,\varepsilon_1,\varepsilon_2,\ldots,\varepsilon_{k}])\\
  =  \sum_{a\in\A}\lambda^{-k-1}\mathbf{v}_L^\trans M_aM_{\varepsilon_1}\cdots
  M_{\varepsilon_k}\mathbf{v}_R=
  \lambda^{-k-1}\mathbf{v}_L^\trans MM_{\varepsilon_1}\cdots
  M_{\varepsilon_k}\mathbf{v}_R\\
  =  \lambda^{-k}\mathbf{v}_L^\trans M_{\varepsilon_1}\cdots
  M_{\varepsilon_k}\mathbf{v}_R=
  \mu^+([\varepsilon_1,\varepsilon_2,\ldots,\varepsilon_{k}]).
\end{multline*}
A similar computation shows that the measure is well defined by Kolmogorov's
consistency theorem.

Assuming that $\LL$ is primitive, the dynamical system
$(\K^+,\mu^+,\sigma^+)$ is strongly mixing and thus ergodic (see
\cite{Walters1982:ergodic_theory,Cornfeld_Fomin_Sinai1982:ergodic_theory}):
\begin{multline*}
  \lim_{n\to\infty}\mu^+
  \left([\varepsilon_1,\varepsilon_2,\ldots,\varepsilon_{k}]
    \cap (\sigma^+)^{-n}[\delta_1,\ldots,\delta_s]\right)\\
  = \lim_{n\to\infty}\lambda^{-n-s}\mathbf{v}_L^\trans M_{\varepsilon_1}\cdots
  M_{\varepsilon_k}M^{n-k}M_{\delta_1}\cdots M_{\delta_s}\mathbf{v}_R\\
  =  \lambda^{-k}\mathbf{v}_L^\trans M_{\varepsilon_1}\cdots
  M_{\varepsilon_k}\mathbf{v}_R \lambda^{-s}\mathbf{v}_L^\trans
  M_{\delta_1}\cdots M_{\delta_s}\mathbf{v}_R\\
  =  \mu^+([\varepsilon_1,\varepsilon_1,\ldots,\varepsilon_{k}])
  \mu^+([\delta_1,\ldots,\delta_s]).
\end{multline*}
The measures $\mu_I^+$ and $\mu^+|_{\K_I^+}$ (restriction to $\K_I^+$) are
equivalent. The measure $\mu^+$ is the unique shift invariant measure of
maximal entropy on $\K^+$, also called the Parry measure
\cite{Parry1966:symbolic_dynamics_transformations,
  Parry1964:intrinsic_markov_chains}. Maximality and uniqueness follow from
\cite[Theorems~6 and~7]{Parry1964:intrinsic_markov_chains} together with
\cite[Theorem~10]{Parry1964:intrinsic_markov_chains}.

Similarly, we define a measure on $\K$ by
\begin{equation}
  \label{eq:mudef}
  \mu([\varepsilon_1,\ldots,\varepsilon_k]_m)=
  \lambda^{-k}\mathbf{v}_L^\trans M_{\varepsilon_1}\cdots
    M_{\varepsilon_k}\mathbf{v}_R,
\end{equation}
where
\begin{equation*}
  [\varepsilon_1,\ldots,\varepsilon_k]_m=\left\{(x_n)_{n\in\mathbb{Z}}\in\K
    \,\middle|\,
  x_{m+1}=\varepsilon_1,\ldots,x_{m+k}=\varepsilon_k\right\}
\end{equation*}
for $m\in\mathbb{Z}$. The measure $\mu$ is $\sigma$-invariant by definition.

\section{Bernoulli convolutions}\label{sec:bern-conv}
In \cite{Erdoes39} Erd\H{o}s studied the distribution measure of the random
series
\begin{equation}\label{eq:random-sum}
  \sum_{k=1}^\infty\frac{X_k}{\beta^k},
\end{equation}
where $(X_k)_{k\in\mathbb{N}}$ is a sequence of i.i.d. random variables taking
the values $\pm1$ with equal probability $\frac12$, and
$\beta=\frac{1+\sqrt5}2$. He showed that the distribution is purely singular
continuous. Later \cite{Erdoes40} he extended this result for $\beta$ an
irrational Pisot number.  The Pisot property plays an important r\^ole in the
proof of singularity, as it allows to show that the Fourier transform of the
measure does not tend to $0$ along the sequence
$(\beta^k)_{k\in\mathbb{N}}$. This argument will be elaborated later.

In the meantime the set of $\beta>1$, for which the measure constructed as
above is singular continuous has been studied further. Solomyak
\cite{Solomyak1995:random_series_sum} could prove that for almost all
$\beta\in(1,2)$ the measure is absolutely continuous. This result was refined
by Shmerkin \cite{Shmerkin2014:exceptional_set_bernoulli}, who proved that the
exceptional set has Hausdorff dimension $0$. It is still open, whether Pisot
numbers are the only exceptions. For a survey on the development until the year
2000 we refer to \cite{Peres_Schlag_Solomyak2000:sixty_years_bernoulli}. For
more recent developments and results in this direction we refer to
\cite{Saglietti_Shmerkin_Solomyak2018:absolute_continuity,
  Shmerkin_Solomyak2016:absolute_continuity_self}.

A newer development in the study of Bernoulli convolutions was initiated with
the proof that ergodic invariant measures on the full shift
$\mathcal{A}^{\mathbb{N}}$ are projected to exact dimensional measures by
iterated function systems (see
\cite{Feng_Hu2009:dimension_theory_iterated}). More precisely, a measure $\nu$
is called exact dimensional with dimension $\alpha$, if
\begin{equation*}
  \lim_{r\to0}\frac{\log\nu((x-r,x+r))}{\log r}=\alpha
\end{equation*}
for $\nu$-almost all $x$. Of course, absolutely continuous measures on
$\mathbb{R}$ have dimension $1$, whereas the opposite implication is not
true. This allowed for the proof that the dimension of the distribution measure
of \eqref{eq:random-sum} has dimension $1$ for all transcendental
$\beta\in(1,2)$ (see \cite{Varju2019:dimension_bernoulli_convolutions}). For
further results in that direction we refer to
\cite{Breuillard_Varju2019:dimension_bernoulli_convolutions,
  Varju2019:absolute_continuity_bernoulli,Varju2018:recent_progress_bernoulli}.

\section{Generalised Erd\H{o}s measures}
\label{sec:gener-erdh-meas}
From now on we assume that the alphabet $\A$ is a subset of $\mathbb{Z}$.  For
an automaton $\auto$ and a Pisot number $\beta$ of degree $r\geq1$ we introduce
the maps
\begin{equation}
  \label{eq:phi-def}
  \begin{split}
    \phi_{\mathcal{L}}^+&:\mathcal{L}\to\mathbb{R}: (x_k)_{k=1}^n\mapsto
    \sum_{k=1}^n\frac{x_k}{\beta^k}\\
    \phi_I^+&:\K_I^+\to\mathbb{R}: (x_k)_{k\in\mathbb{N}}\mapsto\sum_{k=1}^\infty
    \frac{x_k}{\beta^k}\\
    \phi^+&:\K^+\to\mathbb{R}: (x_k)_{k\in\mathbb{N}}\mapsto\sum_{k=1}^\infty
    \frac{x_k}{\beta^k}.
  \end{split}
\end{equation}

The measures
\begin{equation}
  \label{eq:nu-def}
  \begin{split}
    \nu_I&=\left(\phi_I^+\right)_*(\mu_I^+)\\
    \nu&=\left(\phi^+\right)_*(\mu^+)
  \end{split}
\end{equation}
are analogues of the Erd\H{o}s measures studied in
\cite{Erdoes40,Erdoes39}. Here and throughout this paper we denote by
$f_*(\mu)$ the push-forward measure on $Y$ given by a map $f:X\to Y$ and a
measure $\mu$ on $X$; $f_*(\mu)(A)=\mu(f^{-1}(A))$. The properties of the
measures $\nu$ and $\nu_I$ will be the subject of the remaining part of this
paper. By definition, $\nu_I$ is absolutely continuous with respect to
$\nu$. Furthermore, by the definition of $\mu_I^+$, $\nu_I$ is given by
\begin{equation*}
  \nu_I=\lim_{n\to\infty}\frac1{\#\LL_n}\sum_{w\in\LL_n}
  \delta_{\phi_{\mathcal{L}}^+(w)},
\end{equation*}
where $\delta_x$ denotes a unit point mass in $x$.

\begin{theorem}\label{thm1}
  The measures $\nu_I$ and $\nu$ are pure in the sense that they are either
  absolutely continuous with respect to Lebesgue measure, purely singular
  continuous, or purely atomic. The last case can only occur, if the image
  $\phi^+(\K^+)$ is finite. The number of atoms is bounded by the number of
  vertices in $\auto$.
\end{theorem}

\begin{proof}
    The Jessen-Wintner theorem
  \cite[Theorem~35]{Jessen_Wintner1935:distribution_functions_riemann} (for a
  more modern formulation see also
  \cite[Lemma~1.22 (ii)]{Elliott1979:probabilistic_number_theory_I})
  is concerned with  the distribution measure of a random series
  \begin{equation*}
    \sum_{n=1}^\infty X_n,
  \end{equation*}
  where $(X_n)_{n\in\mathbb{N}}$ is a sequence of independent discrete random
  variables and the series
  \begin{equation*}
    \sum_{n=1}^\infty \mathbb{E}(X_n)\text{ and }
    \sum_{n=1}^\infty \mathbb{V}(X_n)
  \end{equation*}
  converge. It states that this measure is either absolutely continuous with
  respect to Lebesgue measure, purely singular continuous, or purely atomic.

  For our purposes we need a more general version, which allows for some
  dependence between the random variables $X_n$.

   \begin{lemma}\label{lem-jessen-wintner}
     Let $\K^+\subset\A^{\N}$ be a shift invariant subset equipped with a shift
     invariant measure $\mu$ such that the shift is ergodic with respect to
     $\mu$. Let
     \begin{equation*}
       X=\sum_{n=1}^\infty X_n
     \end{equation*}
     be a series of random variables $X_n$, where $X_n$ only depends on the
     $n$-th coordinate of the argument. Also assume that the series converges
     $\mu$-almost everywhere. Then the distribution of $X$ is either
     purely discrete, or purely singular continuous, or absolutely continuous
     with respect to Lebesgue measure.
\end{lemma}
\begin{proof}[Proof of Lemma~\ref{lem-jessen-wintner}]
  The proof of this lemma is just the observation that the proof of the
  Jessen-Wintner theorem only uses the fact that the measure on the product
  space satisfies a $0$-$1$-law. The proof follows the lines of proof given for
  \cite[Lemma~1.22(ii)]{Elliott1979:probabilistic_number_theory_I}.
  Let $\nu=X_*(\mu)$ be the distribution measure
  of $X$. Then $\nu$ has a Lebesgue decomposition
  \begin{equation*}
    \nu=a_1\nu_1+a_2\nu_2+a_3\nu_3,
  \end{equation*}
  where $a_1+a_2+a_3=1$ and $\nu_i$ ($i=1,2,3$) are probability measures, where
  $\nu_1$ is purely atomic, $\nu_2$ is absolutely continuous, and $\nu_3$ is
  singular continuous. Let $D$ be the support of $\nu_1$. Then $X^{-1}(D)$ is a
  shift invariant subset of $\K^+$, which has measure $0$ or $1$ by
  ergodicity. Thus $a_1$ is either $0$ or $1$. Similarly, let $S$ be the set of
  Lebesgue measure $0$ that supports $\nu_3$. Then again $X^{-1}(S)$ is shift
  invariant, and again has measure $0$ or $1$. Thus the decomposition has
  exactly one term.
\end{proof}
The statement for $\nu_I$ follows from the absolute continuity of $\nu_I$
with respect to $\nu$. The assertion about the number of atoms will be proved
in the proof of Theorem~\ref{thm2}.
\end{proof}
\begin{lemma}\label{lem2}
  Let $P=\{x\in\mathbb{R}\mid \nu(\{x\}))>0\}$. Then $P$ is a
  finite subset of  $\mathbb{Q}(\beta)$.
\end{lemma}
\begin{proof}
  For $x\in P$ set $A_x=(\phi^+)^{-1}(\{x\})$. Then $\phi^+(A_x)=\{x\}$ and
  \begin{equation*}
    \phi^+(\sigma^{-m}(A_x))\subseteq
    \left\{\beta^{-m}\left(x+\sum_{k=0}^{m-1}\ell(e_{m-k})\beta^k\right)
      \,\middle|\,
    e_1\ldots e_m\text{ a path in }\auto\right\}.
\end{equation*}
Since $\mu^+(A_x)>0$ there is an $n\geq1$ such that
$A_x\cap\sigma^{-n}(A_x)\neq\emptyset$. This implies that
\begin{equation*}
  x\in\left\{\beta^{-n}\left(x+\ell(e_n)+\cdots+
        \ell(e_1)\beta^{n-1}\right)\,\middle|\,
    e_1\ldots e_n\text{ a path in }\auto\right\},
\end{equation*}
from which it follows that $x\in\mathbb{Q}(\beta)$. Then there exists
$N\in\mathbb{N}$ such that $x\in\frac1N\mathbb{Z}[\beta]$.

Now fix $x\in P$ and thus $N$. Let $y\in P$. Then there exists an $n\geq1$
such that $A_x\cap\sigma^{-n}(A_y)\neq\emptyset$ again by ergodicity. This
implies that there exists a path $e_1\ldots e_n$ in $\auto$ such that
\begin{equation*}
  y=\beta^nx-\sum_{k=0}^{n-1}\ell(e_{n-k})\beta^k,
\end{equation*}
which shows that $y\in\frac1N\mathbb{Z}[\beta]$. Thus
$P\subset\frac1N\mathbb{Z}[\beta]$. Now by definition $P$ is bounded. Applying
the conjugations $\beta\mapsto\beta_q$ ($q=2,\ldots,r$) gives
\begin{equation*}
  |y_q|\leq |\beta_q|^n|x_q|+\sum_{k=0}^{m-1}|\ell(e_{m-k})||\beta_q|^k\leq
  |x_q|+\frac M{1-|\beta_q|},
\end{equation*}
if $|a|\leq M$ for all $a\in\A$. Thus $NP$ is a set of algebraic integers all
of whose conjugates are bounded. Thus $P$ is finite.
\end{proof}
\begin{lemma}\label{lem3}
  Let $\beta$ be a Pisot number.
  Then the set of words
  \begin{equation*}
    \LL_0=\left\{(x_1,x_2,\ldots,x_n)\in\A^*\,\middle|\,
      \sum_{k=1}^n\frac{x_k}{\beta^k}=0\right\}
  \end{equation*}
  is recognisable by a finite automaton. As a consequence the set
  \begin{equation}\label{eq:K0}
    \K_0^+=\left\{(x_1,x_2,\ldots)\in\A^{\mathbb{N}}\,\middle|\,
      \sum_{k=1}^\infty\frac{x_k}{\beta^k}=0\right\}
  \end{equation}
  is a space $\K_I^+$ for that automaton and an appropriate initial state $I$
  (labelled by $0$).
\end{lemma}
\begin{proof} Assume that $\LL_0\neq\emptyset$; otherwise the assertion is
  trivial.
  We define the set
  \begin{equation*}
    E=\left\{\sum_{k=1}^{m-s}\frac{x_{k+s}}{\beta^k}\middle|
    0\leq s<m, (x_1,\ldots,x_m)\in\LL_0\right\}.
  \end{equation*}
  This will be the set of states. Between two elements $x,y\in E$ there is a
  transition, if and only if
  \begin{equation*}
    a=\beta x-y\in\A;
  \end{equation*}
  this transition will then be marked with $a$. Since by definition of $E$
  every element of $E$ can be obtained from $0\in E$ by finitely many
  transitions $x\mapsto\beta x-a$ with $a\in\A$. This shows that every element
  of $E$ can be expressed in the form
  \begin{equation*}
    -\sum_{k=0}^mx_k\beta^k\quad\text{with }x_0,\ldots,x_m\in\A.
  \end{equation*}
  This shows that all elements of $E$ are algebraic integers. The original
  definition of $E$ shows that all elements of $E$ are bounded by $\frac
  M{\beta-1}$, if $|a|\leq M$ for all $a\in\A$. Furthermore, all conjugates of
  elements of $E$ are bounded by
  \begin{equation*}
    \left|-\sum_{k=0}^mx_k\beta_q^k\right|\leq \frac M{1-|\beta_q|}
  \end{equation*}
  for $q=2,\ldots,r$. Thus the set $E$ consists of algebraic integers all of
  whose conjugates are bounded; thus $E$ is finite. The set $E$ together with
  the transitions defined above define a finite automaton recognising $\LL_0$
  by taking $I=\{0\}$ as initial and terminal state.

  All infinite words recognised by this automaton correspond to sums as in the
  definition of $\K_0$.
\end{proof}
\begin{lemma}\label{lem4}
  Let $x\in\mathbb{R}$ be such that $\nu(\{x\}))>0$.
  Then the set $(\phi^+)^{-1}(\{x\})$ is open.
\end{lemma}
\begin{proof}
  Since the set $P$ of atoms of $\nu$ is finite,
  $(\phi^+)^{-1}(\{x\})=(\phi^+)^{-1}((x-\varepsilon,x+\varepsilon))$ for small
  enough $\varepsilon>0$. The continuity of $\phi^+$ implies that
  $(\phi^+)^{-1}(\{x\})$ is open.
\end{proof}
\begin{theorem}\label{thm2}
  The set $\phi^+(\K^+)$ is either finite or perfect and thus uncountable. In
  the first case the measure $\phi^+_*(\mu^+)$ is atomic, in the second case it
  is continuous.
\end{theorem}
\begin{proof}
  The set $\phi^+(\K^+)$ is compact as the continuous image of a compact set.
  
  Assume that there exists a vertex $v\in V$ and two paths $e_1e_2\ldots
  e_{L_1}$ and $f_1f_2\ldots f_{L_2}$ both connecting $v$ to itself such that
  \begin{equation}\label{eq:different}
    \frac1{1-\beta^{-L_1}}\sum_{k=1}^{L_1}\frac{\ell(e_k)}{\beta^k}\neq
    \frac1{1-\beta^{-L_2}}\sum_{k=1}^{L_2}\frac{\ell(f_k)}{\beta^k}.
  \end{equation}
  Then we will show that the set $\phi^+(\K^+)$ is perfect. For this purpose we
  choose $x\in\phi^+(\K^+)$ and show that there is a sequence
  $(x_n)_{n\in\mathbb{N}}$ of points in $\phi^+(\K^+)$ with
  $x=\lim_{n\to\infty}x_n$ and $x_n\neq x$ for all $n$.

  Let $x=\phi^+((\ell(a_1),\ell(a_2),\ldots))$. For $n\in\mathbb{N}$ we set
  \begin{align*}
    \xi_n&=\phi^+(\ell(a_1),\ell(a_2),\ldots,\ell(a_n),\ell(b_1),\ldots,
    \ell(b_k),\overline{\ell(e_1),\ldots,\ell(e_{L_1})})\\
    \eta_n&=\phi^+(\ell(a_1),\ell(a_2),\ldots,\ell(a_n),\ell(b_1),\ldots,
    \ell(b_k),\overline{\ell(f_1),\ldots,\ell(f_{L_2})}),
  \end{align*}
  where $b_1,\ldots,b_k\in E$ are chosen so that
  $\ell(a_1)\ldots\ell(a_n)\ell(b_1)\ldots\ell(b_k)\in\LL$ and
  $\mathrm{t}(b_k)=v$. There exists an integer $k\leq\#V$ with this property
  for every $n$. By \eqref{eq:different} $\xi_n\neq\eta_n$, and thus at least
  one of these two values is different from $x$. We take $x_n$ to be this
  value. Then $\lim_{n\to\infty}x_n=x$ showing that $x$ is not isolated.

  By Lemma~\ref{lem4} the preimage of any atom of the measure $\phi^+_*(\mu^+)$
  would contain a cylinder set. This would contradict the fact that we have
  just proved that the images of all cylinder sets are uncountable. This shows
  that $\nu$ is continuous in this case.

  Assume on the contrary that for all $v\in V$ there exists a value
  $c(v)\in\mathbb{R}$ such that for all paths $e_1e_2\ldots e_n$ connecting $v$
  to itself
  \begin{equation}
    \label{eq:cv}
    c(v)=\frac1{1-\beta^{-n}}\sum_{k=1}^n\frac{\ell(e_k)}{\beta^k}.
  \end{equation}
  In this case every infinite path $e_1e_2\ldots$ starting at $v$ yields the
  value $c(v)$ for $\phi^+$: assume that $j$ is chosen to be the minimal index
  so that $w=\mathrm{t}(e_j)$ is visited infinitely often by the path. If
  $j=1$, the value of $\phi^+$ given by the path is $c(v)$ by definition. For
  $j>1$ we decompose the path
  \begin{equation*}
    e_1\ldots e_{k_1}e_{k_1+1}\ldots e_{k_2}e_{k_2+1}\ldots e_{k_3}\ldots,
  \end{equation*}
  where $k_1=j$ and $\mathrm{t}(e_{k_m})=w$. Then by our assumption
  \eqref{eq:cv} every path $e_{k_m+1}\ldots e_{k_{m+1}}$ can be replaced by a
  path $f_1\ldots f_sf_{s+1}\ldots f_{s+q}$ with $\mathrm{i}(f_1)=w$,
  $\mathrm{t}(f_s)=\mathrm{i}(f_{s+1})=v$, and $\mathrm{t}(f_{s+q})=w$ without
  changing the value of $\phi^+$. The new path visits $v$ infinitely often, and
  thus assigns the value $c(v)$ to $\phi^+$. Thus $\phi^+$ only takes the
  values $\{c(v)\mid v\in V\}$. Each of these values is assigned a positive
  mass. This proves that the number of atoms of $\phi^+_*(\mu)$ is bounded by
  the number of vertices of $\auto$. This shows the last assertion of
  Theorem~\ref{thm1}.
\end{proof}
The following result is a consequence of the proof of Theorem~\ref{thm2}.
\begin{theorem}\label{thm8}
  The set $\phi^+(\K^+)$ is finite, if and only if for every vertex $v\in V$
  there exists a value $c(v)\in\R$ such that for all paths $e_1e_2\ldots e_n$
  connecting $v$ to itself \eqref{eq:cv} holds.  
\end{theorem}
\begin{remark}
  Notice that the measure $\phi_*^+(\mu^+)$ can only be absolutely continuous,
  if $\beta\leq\lambda$ (recall that $\lambda$ is the Perron-Frobenius
  eigenvalue associated to the automaton $\auto$). Otherwise the set
  $\phi^+(\K^+)$ has Hausdorff dimension $\leq \frac{\log\lambda}{\log\beta}<1$
  and cannot support an absolutely continuous measure.
\end{remark}
\section{Fourier transforms and matrix products}
\label{sec:four-transf-matr}
The Fourier transforms of the measures $\nu$ and $\nu_I$ are given by
\begin{equation}
  \label{eq:fourier}
  \begin{split}
    \nuhat_I(t)=\int\limits_{-\infty}^\infty e^{-2\pi ixt}\,d\nu_I(x)\\
    \nuhat(t)=\int\limits_{-\infty}^\infty e^{-2\pi ixt}\,d\nu(x).
  \end{split}
\end{equation}

In order to derive expressions for $\nuhat_I$ and $\nuhat$, we introduce the
weighted transition matrix for $\phi^+$ and the underlying automaton $\auto$
\begin{equation}
  \label{eq:W-def}
  W(t)=\frac1\lambda\sum_{a\in\A}e(-at)M_a,
\end{equation}
where we use the notation $e(x)=e^{2\pi ix}$.

\begin{proposition}
  The Fourier transforms of the measures $\nu$ and $\nu_I$ can be expressed as
  \begin{equation}
    \label{eq:nuhat}
    \nuhat(t)=\mathbf{v}_L^\trans\prod_{n=1}^\infty W(\beta^{-n}t)\mathbf{v}_R
  \end{equation}
  and
  \begin{equation}
    \label{eq:nuhatI}
    \nuhat_I(t)=\frac1{\mathbf{v}_I^\trans\mathbf{v}_R}
    \mathbf{v}_I^\trans\prod_{n=1}^\infty W(\beta^{-n}t)\mathbf{v}_R,
  \end{equation}
  where the infinite matrix product is interpreted so that the factors are
  ordered from left to right.
\end{proposition}
\begin{proof}
  We set
  \begin{equation*}
    \phi_n^+((x_k)_{k\in\N})=\sum_{k=1}^n\frac{x_k}{\beta^k}.
  \end{equation*}
  Then $\lim_{n\to\infty}\phi_n^+((x_k)_{k\in\N})=\phi^+((x_k)_{k\in\N})$ holds
  uniformly on $\K^+$. We set $\nu_n=(\phi_n^+)_*(\mu^+)$ and observe that
  $\nu_n\rightharpoonup\nu$. Then
  \begin{equation*}
    \nuhat_n(t)=\mathbf{v}_L^\trans\prod_{k=1}^n W(\beta^{-k}t)\mathbf{v}_R.
  \end{equation*}
  The limit relation $\lim_{n\to\infty}\nuhat_n=\nuhat$ and equation
  \eqref{eq:nuhat} then follow by weak convergence. A similar reasoning gives
  \eqref{eq:nuhatI}.
\end{proof}
\begin{remark}
  As pointed out in Section~\ref{sec:bern-conv} Erd\H{o}s
  \cite{Erdoes40,Erdoes39} proved the singularity of the distribution measure
  $\nu$ of the random series
  \begin{equation*}
    \sum_{n=1}^\infty\frac{X_n}{\beta^n},
  \end{equation*}
  by showing that $(\nuhat(\beta^k))_{k\in\mathbb{N}}$ does not tend to $0$. Of
  course the fact that $\lim_{|t|\to\infty}\nuhat(t)=0$ does not suffice in
  general to prove absolute continuity, as there are singular measures, so
  called Rajchman measures, whose Fourier transform vanishes at $\infty$ (see
  \cite{Rajchman1929:une_classe_series}).
\end{remark}
Using the Pisot property of $\beta$ we define the map
\begin{equation}
  \label{eq:phidef}
    \Phi:\K\to\T^r,
    (x_k)_{k\in\mathbb{Z}}\mapsto\left(\sum_{k=-\infty}^\infty
      x_k\beta^{-k+m}\pmod1 \right)_{m=0}^{r-1},
  \end{equation}
  where we use the notation $\T=\R/\Z$.
Notice that the series for $k\leq0$ converges $\pmod1$ by the fact that
\begin{equation*}
  \beta^m+\beta_2^m+\cdots+\beta_r^m\in\mathbb{Z}
\end{equation*}
and $|\beta_2|,\ldots,|\beta_r|<1$.
\begin{remark}\label{rem:Sidorov}
  A similar map was studied in \cite{Sidorov2001:bijective_general_arithmetic,
    Sidorov_Vershik1998:ergodic_properties_erdoes}. In
  \cite{Sidorov_Vershik1998:ergodic_properties_erdoes} this map was used to
  give a dynamic proof of the singularity of the Erd\H{o}s measure introduced
  in \cite{Erdoes39}. In \cite{Sidorov2001:bijective_general_arithmetic}
  conditions were given under which this map between the two-sided
  $\beta$-shift $X_\beta$ and $\mathbb{T}^d$ is almost bijective. In the
  present situation, where $\beta$ and the underlying language $\mathcal{L}$
  are not necessarily related, nothing can be said about (almost) bijectivity
  of this map.
\end{remark}

\begin{theorem}
  Let $z=m_0+m_1\beta+\cdots+m_{r-1}\beta^{r-1}\in\mathbb{Z}[\beta]$ for
  $\beta$ a Pisot number of degree $r$ and $\nu$ be the measure given by
  \eqref{eq:nu-def}. Then the limit
  \begin{equation}
    \label{eq:psihat-def}
    \psihat(m_0,\ldots,m_{r-1})=\lim_{k\to\infty}\nuhat(z\beta^k)
  \end{equation}
  exists. These values are the Fourier coefficients of the measure
  $\psi=\Phi_*(\mu)$ on $\T^r$.
\end{theorem}
\begin{proof}
  We define the maps
  \begin{equation*}
    \Phi_n((x_k)_{k\in\Z})=\left(\sum_{k=-n}^\infty x_k\beta^{-k+m}
      \pmod1\right)_{m=0}^{r-1}.
  \end{equation*}
  Then $(\Phi_n)_{n\in\N}$ converges to $\Phi$ uniformly on $\K$. Let
  $\psi_n=(\Phi_n)_*(\mu)$. Then $\psi_n\rightharpoonup\psi$ and
  \begin{equation*}
    \psihat_n(m_0,\ldots,m_{r-1})=\nuhat(z\beta^n).
  \end{equation*}
  The limit relation \eqref{eq:psihat-def} then follows by weak convergence and
  the fact that $(\beta^n\mod1)\to0$.
\end{proof}
\begin{remark}
  The shift on $\K$ is conjugate via $\Phi$ to the hyperbolic toral
  endomorphism
  \begin{equation*}
    B:\T^r\to\T^r, \begin{pmatrix}
      t_0\\t_1\\\vdots\\t_{r-1}
    \end{pmatrix}
    \mapsto
    \begin{pmatrix}
      t_1\\\vdots\\t_{r-1}\\a_0t_0+\cdots+a_{r-1}t_{r-1}\pmod1
    \end{pmatrix},
  \end{equation*}
  where
  \begin{equation*}
    \beta^r=a_{r-1}\beta^{r-1}+\cdots+a_1\beta+a_0
  \end{equation*}
  is the minimal equation of $\beta$. The measure $\psi$ is then a
  $B$-invariant measure on $\T^r$, and $B$ is ergodic with respect to $\psi$.
\end{remark}
\begin{theorem}\label{thm:Rajchman}
  The measure $\nu$ given by \eqref{eq:nu-def} is absolutely continuous, if and
  only if for all $z\in\mathbb{Z}[\beta]\setminus\{0\}$
  \begin{equation}\label{eq:limnuhat0}
    \lim_{k\to\infty}\nuhat(z\beta^k)=0.
  \end{equation}
\end{theorem}
\begin{proof}
  Let $\beta_2,\ldots,\beta_r$ denote the Galois conjugates of $\beta$ and
  assume that $\beta,\beta_2,\ldots,\beta_s\in\mathbb{R}$ and
  $\beta_{s+1},\beta_{s+2},\ldots,\beta_{s+t},\overline{\beta_{s+1}},
  \overline{\beta_{s+2}},\ldots,\overline{\beta_{s+t}}
  \in\mathbb{C}\setminus\mathbb{R}$; then $r=s+2t$. Then the map
  \begin{equation*}
    \widetilde{\Phi}:\K\to\mathbb{R}^s\times\mathbb{C}^t,
    (x_k)_{k\in\mathbb{Z}}\mapsto
    \left(\sum_{k=1}^\infty x_k\beta^{-k},\sum_{k=0}^\infty x_{-k}\beta_2^k,
    \ldots,\sum_{k=0}^\infty x_{-k}\beta_{s+t}^k\right)
\end{equation*}
is continuous and $\widetilde{\Phi}(\K)$ is compact. Together with the map
\begin{multline*}
  \rho:\mathbb{R}^s\times\mathbb{C}^t\to\mathbb{R}^r,
  (y_1,\ldots,y_s,z_{s+1},\ldots,z_{s+t})\mapsto\\
  \left(\beta^my_1-(\beta_2^my_2+\cdots+\beta_s^my_s)-
    2\Re(\beta_{s+1}^mz_{s+1}+\cdots+\beta_{s+t}^mz_{s+t})
  \right)_{m=0}^{r-1}
\end{multline*}
we have
\begin{equation*}
  \Phi=\rho\circ\widetilde{\Phi}\pmod1.
\end{equation*}
The map $\rho$ is linear and $\rho\pmod1$ is finite to one on
$\widetilde{\Phi}(\K)$. Now the measure
$\psi=\!\!\pmod1_*\circ\rho_*\circ\widetilde{\Phi}_*(\mu)$ is equal to the
Lebesgue measure, if and only if $\psihat(m_0,\ldots,m_{r-1})=0$ for all
$(m_0,\ldots,m_{r-1})\in\mathbb{Z}^r\setminus\{\mathbf{0}\}$. If $\psi$ is
Lebesgue measure, then the measure $\nu=P_*\circ\widetilde{\Phi}_*(\mu)$ is
absolutely continuous, where $P$ denotes the projection to the first
coordinate. On the other hand, if $\psi$ is not the Lebesgue measure, then
$\nu$ cannot be absolutely continuous by \eqref{eq:psihat-def}.
\end{proof}
\section{Examples}\label{sec:examples}
\begin{example}
  Let $\K\subset\{0,\pm1\}^{\mathbb{N}}$ be given by the automaton in
  Figure~\ref{fig:fib0} and take $\beta=\frac{1+\sqrt{5}}2$. Then the map
  $\phi^+$ takes the values
  \begin{align*}
    &0&&\nu(\{0\})=\frac1{\gamma^2}\\
    \pm&1&&\nu(\{1\})=\nu(\{-1\})=
    \frac12\left(\frac1\gamma-\frac1{\gamma^2}\right)\\
    \pm&\frac1\beta&&\nu\left(\left\{\frac1\beta\right\}\right)
    =\nu\left(\left\{-\frac1\beta\right\}\right)=\frac1{\gamma^3}
  \end{align*}
  with the indicated probabilities, where $\gamma$ is the positive solution of
  \begin{equation*}
    x^3=x^2+2.
  \end{equation*}
  The value $\gamma$ is the Perron-Frobenius eigenvalue of the adjacency matrix
  of the automaton given by Figure~\ref{fig:fib0}.
  \begin{figure}
    \centering
    \includegraphics[width=0.4\hsize]{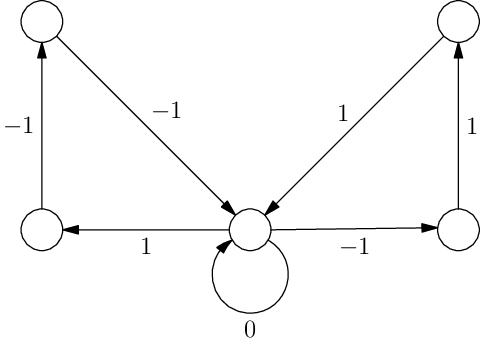}
    \caption{\label{fig:fib0}
      The automaton recognising all expansions of $0$ in base
      $\frac{1+\sqrt{5}}2$ with digits $\{0,\pm1\}$}
  \end{figure}
\end{example}
\begin{example}
  Let $\K\subset\{0,1\}^{\mathbb{N}}$ be the set of sequences of $0$ and $1$
  with no two consecutive $1$s (given by the automaton in
  Figure~\ref{fig:fibo}). These are the digital representations of all real
  numbers in $[0,1]$ obtained by iteration of the $\beta$-transformation
  $T_\beta:x\mapsto\beta x\mod1$ for $\beta=\frac{1+\sqrt{5}}2$
  (see~\cite{Parry1960:beta_expansions_real}). Then the measure $\phi^+(\mu^+)$
  is Lebesgue measure on $[0,1]$.
  \begin{figure}
    \centering
    \includegraphics[width=0.2\hsize]{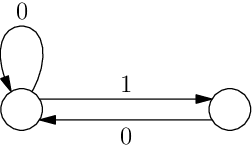}
    \caption{\label{fig:fibo}
      The automaton recognising all greedy expansions in base
      $\frac{1+\sqrt{5}}2$ with digits $\{0,1\}$}
  \end{figure}
\end{example}
\begin{example}\label{ex2}
  Take $\K=\{0,1,2,3\}^\N$ and $\mu$ the infinite product measure assigning
  probability $\frac14$ to each letter. Take $\beta=2$ and consider the map
  $\phi^+$ as above. Then the corresponding measure on $[0,3]$ is given by the
  density
  \begin{equation*}
    h(x)=
    \begin{cases}
      \frac12x&\text{for }0\leq x\leq 1\\
      \frac12&\text{for }1\leq x\leq2\\
      \frac12(3-x)&\text{for }2\leq x\leq3\\
      0&\text{otherwise.}
    \end{cases}
  \end{equation*}
  Projecting this measure $\mod 1$ gives Lebesgue measure, because in this
  case $r=1$ and the map $\Phi$ maps to $\T$.
\end{example}
\begin{example}
  Let $\K^+\subset\{0,1,2\}^\N$ be the set of sequences recognised by the
  automaton in Figure~\ref{fig:fibo2}. These are all expansions of real numbers
  expressed in base $\beta^2$, where $\beta=\frac{1+\sqrt5}2$. This is similar
  to Example~\ref{ex2}, where expansions in base $\beta^2$ are interpreted in
  base $\beta$. In this case the measure is singular, though. This can be seen
  by computing
  \begin{equation*}
    \lim_{n\to\infty}\nuhat(\beta^n)=0.0608424\ldots + 0.0208583\ldots i.
  \end{equation*}
  numerically.
    \begin{figure}
    \centering
    \includegraphics[width=0.2\hsize]{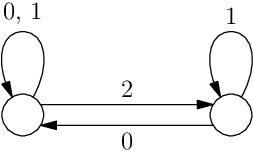}
    \caption{\label{fig:fibo2}The automaton recognising all expansions in base
      $\beta^2=\frac{3+\sqrt{5}}2$ with digits $\{0,1,2\}$}
  \end{figure}
\end{example}
\begin{example}
  This example is taken from
  \cite{Grabner_Heuberger2006:number_optimal_base}. Take
  $\K\subset\{0,\pm1\}^\N$ to be the set of sequences recognised by the
  automaton in Figure~\ref{fig:min}. The automaton with initial state $I$
  recognises all expansions of integers in the form
  \begin{equation*}
    n=\sum_{k=0}^Kx_k2^k\quad\text{with }x_k\in\{0,\pm1\}
  \end{equation*}
  which minimise the weight
  \begin{equation*}
    w(n)=\sum_{k=0}^K|x_k|.
  \end{equation*}
  This is motivated by the fact that this weight is the number of
  additions/subtractions in computing $nP$ by a Horner-type scheme, where $P$
  is a point on an elliptic curve. The measures $\phi_*^+(\mu^+)$ and
  $\phi_*^+(\mu_I^+)$ are singular in this case, as was proved in
  \cite{Grabner_Heuberger2006:number_optimal_base}.
  \begin{figure}
    \centering
    \includegraphics[width=0.8\hsize]{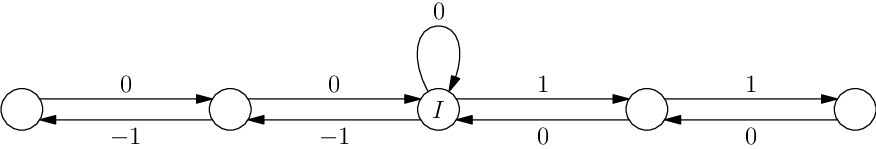}
    \caption{\label{fig:min}
      The automaton recognising expansions in base $2$ with digits
      $\{0,\pm1\}$}
  \end{figure}
\end{example}

\begin{ackno}
  The author is indebted to J\"org~Thuswaldner for very valuable discussions on
  the subject (several years ago in a caf\'e in Pisa).\\
  The author is grateful to an anonymous referee for her/his valuable comments
  that increased the presentation of this paper.
\end{ackno}
\bibliographystyle{amsplain}
\bibliography{refs}
\end{document}